\documentclass[11pt]{amsart}
\usepackage{mathtools} \mathtoolsset{showonlyrefs,showmanualtags}
\usepackage{hyperref} %,hypertexnames=false,colorlinks,[pagebackref]
\hypersetup{
    colorlinks=true,       % false: boxed links; true: colored links
    linkcolor=blue,          % color of internal links
    citecolor=magenta,        % color of links to bibliography
    filecolor=magenta,      % color of file links
    urlcolor=cyan           % color of external links
}

\usepackage{xcolor}
\numberwithin{equation}{section}
\allowdisplaybreaks[4]
\usepackage{amsmath,amsthm,amscd,amssymb,amsfonts, amsbsy}
\usepackage{latexsym}
\usepackage{exscale}
\newtheorem{Theorem}[equation]{Theorem}

\newtheorem{Lemma}[equation]{Lemma}

\def\XXint#1#2#3{{\setbox0=\hbox{$#1{#2#3}{\int}$}
\vcenter{\hbox{$#2#3$}}\kern-.5\wd0}}

\def\bbR{\mathbb{R}}

\def\ch{\rm{ch}}
\def\dw{\textup{d}w}
\begin{document}

\title[On $A_p$-$A_\infty$ type estimates for square functions]{On $A_p$-$A_\infty$ type estimates for square functions}

\author[M. Lacey]{Michael T. Lacey}   %  can use \and

\address{ School of Mathematics, Georgia Institute of Technology, Atlanta GA 30332, USA}
\email {lacey@math.gatech.edu}
\thanks{The first author was supported in part by grant NSF-DMS 1265570. }

\author{Kangwei Li}
\address{School of Mathematical Sciences and LPMC,  Nankai University,
      Tianjin~300071, China}
\email{likangwei9@mail.nankai.edu.cn}

\date{\today}


\begin{abstract}
We prove strong-type $A_p$-$A_\infty$ estimate for square functions, improving on the $ A_p$ bound  due to Lerner.
Entropy bounds, in the recent innovation of Treil-Volberg, are then proved.
The techniques of proof include parallel stopping cubes, pigeon-hole arguments, and the approach to entropy
bounds of Lacey--Spencer.
\end{abstract}

\maketitle

%%%%%%%%%%%%%%%%%%%%%%%%%%%%%%%
\section{Introduction}

What are the weakest `$ A_p$ like' conditions that are sufficient for two weight inequalities for square functions?
Replacing square functions by  Calder\'on-Zygmund operators,
this question has received wide ranging attention since the birth of the weighted theory.
The  finest results known are  the $ A_p$-$A_ \infty $ bounds \cite{HL};  the mixed $A_p$-$A_r$ inequalities of Lerner \cite{Lerner3}, and the
entropy bounds of Treil-Volberg \cite{TV}, and the weak-type bounds of \cite{DLR}.  We refer the reader to the introductions
of these papers for a guide to the long history of this question.
The analog of these results for strong-type bounds for
square functions are the focus of this paper.   (The reference \cite{DLR} also includes weak-type estimates for square functions.)

We begin with the definition of the intrinsic square functions introduced by Wilson \cite{Wilson}. For $0<\alpha\le 1$, let $\mathcal{C}_\alpha$ be the family of functions supported in $\{x: |x|\le 1\}$, satisfying $\int \varphi=0$, and such that for all $x$ and $x'$,
$|\varphi(x)-\varphi(x')|\le |x-x'|^\alpha$. If $f\in L_{\rm{loc}}^1(\bbR^n)$ and
$(y,t)\in \bbR_+^{n+1}$, we define
\[
  A_\alpha(f)(y,t)=\sup_{\varphi\in\mathcal{C}_\alpha}|f*\varphi_t(y)|.
\]
Then the intrinsic square function is defined by
\[
  G_{\beta,\alpha}(f)(x)=\bigg(\int_{\Gamma_\beta(x)}(A_\alpha(f)(y,t))^2\frac{dy dt}{t^{n+1}}\bigg)^{1/2},
\]
where $\Gamma_\beta(x)=\{(y,t): |y-x|<\beta t\}$. If $\beta=1$, set $G_{1,\alpha}(f)=G_\alpha(f)$. Wilson showed that
$G_{\beta,\alpha}(f)\sim G_\alpha(f)$ and it dominates the continuous type square functions including the Lusin area function and Littlewood-Paley $g$ function. Therefore, we only focus on $G_\alpha(f)$.

In \cite{Lerner1}, Lerner gave the following estimate
\begin{equation}\label{eq:sharp}
\|G_\alpha \|_{L^p(w)}\le c(G_\alpha, n, p) [w]_{A_p}^{\max\{\frac 12, \frac 1{p-1}\}}.
\end{equation}
 This estimate is sharp in the exponent of $ [w]_{A_p}$
and hence is the square function analog of the $ A_2$ bound for Calder\'on-Zygmund operators,
proved by Hyt\"onen  \cite{Hytonen2012}.

Lerner \cite{Lerner3}, has established  mixed $A_p$-$A_r$ estimates when $p\ge 3$.
These estimates only involve a single supremum to define, and are restricted to the one weight setting.
The interested reader should refer to \cite{Lerner3}.

Our focus is on the two weight  $A_p$-$A_\infty$ type estimates for square functions.
Given a pair of weights $w $ and $\sigma$, define  $\langle w\rangle_Q:=\frac 1 {|Q|} \int_Q w(x) dx$,
\begin{gather*}
[w,\sigma]_{A_p}:=\sup_Q \langle w\rangle_Q \langle \sigma\rangle_Q^{p-1},
\\
\textup{and} \qquad [w]_{A_\infty}:=\sup_Q \frac 1{w(Q)}\int_Q M(w\mathbf 1_Q) dx.
\end{gather*}
Our first result is the following

%%%%%%%%%%%%%%%%%%%%%%
\begin{Theorem}\label{thm:m1}
Given $1<p<\infty$. Let $w $ and $\sigma$ be a pair of weights such that $[w,\sigma]_{A_p}<\infty$ and $w,\sigma\in A_\infty$.
Then
\begin{equation}\label{eq:apainfty}
\|G_\alpha (\cdot \sigma)\|_{L^p(\sigma)\rightarrow L^p(w)} \lesssim
 \begin{cases}
  [w,\sigma]_{A_p}^{\frac 1p}[\sigma]_{A_\infty}^{\frac 1p}, & 1<p\le 2, \\
  [w,\sigma]_{A_p}^{\frac 1p}([w]_{A_\infty}^{\frac 12-\frac 1p}+[\sigma]_{A_\infty}^{\frac 1p}), & p>2.
 \end{cases}
\end{equation}
where the constant $C$ is independent of the weights $w$ and $\sigma$.
\end{Theorem}
%%%%%%%%%%%%%%%%%%%%%%

Specializing this to the one weight case, we have $ [\sigma ] _{A _{\infty } } \lesssim  [\sigma ] _{A_ {p'}} = [w] _{A_p} ^{\frac 1{p-1}}$,
and so  Lerner's bound \eqref{eq:sharp} follows from the Theorem, as can be checked by
 elementary considerations.
This is interesting, since the inequalities \eqref{eq:apainfty} have $ p=2$ as a critical index, while Lerner's bound has
$ p=3$ as the critical index. We find that this reflected in the proof of the result above, with one term  splitting neatly
at $ p=2$, and another splitting at $ p=2$ and at $ p=3$.

Concerning the proof, we will use the common reduction to a positive sparse square function.
In the two weight setting, we have a characterization of the required inequality in terms of
(quadratic) testing assumptions  \cite{Scurry,Hanninen,Vuorinen}.  These conditions are however
difficult to work with.  Instead, we  use the parallel stopping cubes introduced Lacey, Sawyer, Shen and Uriarte-Tuero \cite{LSSU}, as elaborated in the last section of \cite{Hytonen}.

\smallskip
The second topic is to prove entropy bounds for the square function. Here, we are using the recent
innovative approach of Treil-Volberg \cite{TV}, which is an improvement over the
Orlicz norm approach to `bumping', started by C.~P{\'e}rez \cite{P}.  Again, the reader should consult
\cite{TV} for a history of this point of view, and an explanation of why the entropy method is
stronger than that of Orlicz norm approach.

There is a very close connection between the entropy bounds and $ A _{\infty }$ bounds, a feature
exploited by Lacey-Spencer \cite{LS}.  The entropy conditions are given in terms of the `local $ A _{\infty }$'
constant, which is allowed to take arbitrarily large values, at the cost of a logarithmic penalty.
\[
\rho_\sigma (Q)= \frac {\int_Q M(\sigma \mathbf 1_Q)dx}{\sigma(Q)}\quad \mbox{and} \quad \rho_{\sigma,\epsilon}(Q)=\rho_\sigma(Q)\epsilon(\rho_\sigma (Q)),
\]
where $\epsilon$ is a monotonic increasing function on $(1,\infty)$.  The $ \epsilon $ gives the penalty on
a locally large $ A _{\infty }$ constant.  The entropy conditions then come in two forms, one in which
both weights are `bumped' in a multiplicative manner,
\[
\lceil w,\sigma\rceil_{p,\epsilon}:=
\begin{cases}
\sup_Q\langle w\rangle_Q^{\frac 1p} \langle\sigma\rangle_Q^{\frac 1{p'}}\rho_{\sigma,\epsilon}(Q)^{\frac 1p}, & 1<p\le 2,\\
\sup_Q\langle w\rangle_Q^{\frac 1p} \langle\sigma\rangle_Q^{\frac 1{p'}}\rho_{w,\epsilon}(Q)^{\frac 12-\frac 1p}\rho_{\sigma,\epsilon}(Q)^{\frac 1p}, & p>2.
\end{cases}
\]
and the other in an additive or `separated' fashion,
\[
\lfloor w,\sigma\rfloor_{p,\epsilon,\eta}:=
\begin{cases}
\sup_Q \langle w\rangle_Q^{\frac 1p} \langle\sigma\rangle_Q^{\frac 1{p'}}\rho_{\sigma,\epsilon}(Q)^{\frac 1p}, & 1<p\le 2,\\
\sup_Q\langle w\rangle_Q^{\frac 1p} \langle\sigma\rangle_Q^{\frac 1{p'}}(\rho_{w,\eta}(Q)^{\frac 12-\frac 1p}+\rho_{\sigma,\epsilon}(Q)^{\frac 1p}), & p>2,
\end{cases}
\]
where $\eta$ is another monotonic increasing function on $(1,\infty)$.
Now we are ready to state our entropy bounds for square functions.
These inequalities are the analogs of the main results in \cite{LS}, and the proof is along the
lines of that paper.

\begin{Theorem}\label{thm:m2}
Let $(w,\sigma)$ be a pair of weights and $\epsilon,\eta$ be two monotonic increasing functions on $(1,\infty)$. If $\epsilon$ satisfy
\begin{eqnarray*}
\begin{cases}
 \int_1^\infty \frac 1{t\epsilon(t)^{1/p}} dt<\infty,& 1<p\le 2\\
  \int_1^\infty \frac 1{t\epsilon(t)} dt<\infty,& p>2,
 \end{cases}
\end{eqnarray*}
 then there holds
\begin{equation}\label{eq:entropy}
\|G_\alpha(\cdot \sigma)\|_{L^p(\sigma)\rightarrow L^p(w)} \lesssim \lceil w,\sigma\rceil_{p,\epsilon}.
\end{equation}
For any $1<p<\infty$ and $\epsilon, \eta$ satisfy
\begin{eqnarray*}
\begin{cases}
 \int_1^\infty \frac 1{t\epsilon(t)^{1/p}} dt<\infty,& 1<p\le 2\\
 \int_1^\infty \frac 1{t\epsilon(t)^{1/p}}dt +\int_1^\infty \frac 1{t\eta(t)^{\frac 12-\frac 1p}} dt<\infty, & p>2.
\end{cases}
\end{eqnarray*} then there holds
\begin{equation}\label{eq:separated}
\|G_\alpha(\cdot \sigma)\|_{L^p(\sigma)\rightarrow L^p(w)} \lesssim \lfloor w,\sigma\rfloor_{p,\epsilon,\eta}.
\end{equation}
\end{Theorem}

\section{Proof of Theorem~\ref{thm:m1}}
The first step is, as is fundamental in this subject, the reduction to positive sparse operators.
A collection of dyadic cubes $\mathcal S$ is said to be \emph{sparse} if for all $ Q\in \mathcal S$,
\[
\Bigl\lvert \bigcup_{\substack{Q',Q\in \mathcal S\\Q'\subsetneq Q}} Q' \Bigr\rvert \le \frac 12 |Q|.
\]
Then the positive sparse operator related to $\mathcal S$ is defined by
\[
A_{\mathcal S} (f)=\bigg(\sum_{Q\in \mathcal S} \langle f\rangle_{ Q}^2 \mathbf 1_Q\bigg)^{1/2}.
\]
It is well known that there are many dyadic grids, and moreover, there are at most $ 3 ^{n}$ choices of
dyadic grids in $ \mathbb R ^{n}$ so that \emph{any} cube in $ \mathbb R ^{n}$ is well-approximated by
a choice of cube from one of the specified dyadic grids.

The following lemma is a variant of the argument in \cite{L_A2}.
\begin{Lemma}\label{lm:3d}
If $ f$ is bounded and compactly supported, there are at most $ 3 ^{n}$ sparse collections of dyadic cubes $ \mathcal S_j$,
$ 1\le j \le 3 ^{n}$,  so that the pointwise inequality below holds.
\begin{equation*}
G_\alpha f \lesssim \sum_{j=1} ^{3 ^{n}} A _{\mathcal S_j} f .
\end{equation*}
\end{Lemma}

Clearly, we need only consider the weighted bounds for a single sparse operator.  There is then nothing
special about the $ \ell  ^2 $ sum used to define the operator, hence we define $ \ell ^{r}$ variants as follows.
\begin{equation}\label{eq:t}
\Bigl(A_{\mathcal S}^rf \Bigr) ^{r}=\sum_{Q\in\mathcal S} \langle f\rangle_Q^r \mathbf 1_Q.
\end{equation}
We have the following more general estimate.

%%%%%%%%%%%%%%%%%%%%%%%%%%%%%%   Thm
\begin{Theorem}\label{thm:p}
Let $A_{\mathcal S}^r$ be defined as in \eqref{eq:t} and $p >r$, then
\[
\|A_{\mathcal S}^r(\cdot\sigma)\|_{L^p(\sigma)\rightarrow L^p(w)}
\le C [w,\sigma]_{A_p}^{\frac 1p}\bigl([w]_{A_\infty}^{\frac 1r-\frac 1p} + [\sigma ] _{A _{\infty }}  ^{\frac 1p}\bigr)
\]
where the constant $C$ is independent of $w$ and $\sigma$.
\end{Theorem}
%%%%%%%%%%%%%%%%%%%%%%%%%%%%%%

It is a useful remark that the estimate above can be made slightly more precise, in that the supremums defining
the two-weight $ A_p$ and $ A _{\infty }$ constants need only be taken over the collection of cubes $ \mathcal S$.
To be precise, $ [w, \sigma ] _{A_p}$ above can be replaced by
\begin{equation} \label{e:reduce}
[w, \sigma ] _{A_p, \mathcal S} :=
\sup _{Q\in \mathcal S} \langle \sigma  \rangle_Q ^{p-1} \langle w \rangle_Q,
\end{equation}
and similarly for $ [w] _{A_ \infty }$ can be replaced by $ [ w] _{A _{\infty }, \mathcal S}$, which has a similar definition.

%%%%%%%%%%%%%%%%%%%%%%%%%%%%%% PROOF PROOF PROOF
\begin{proof}
Use duality to eliminate the $ r$th root. Since $ p>r$, it suffices to prove
\begin{equation}\label{e:La}
\langle (A ^{r} _{\mathcal S} f \sigma   ) ^{r}, g w \rangle  \lesssim
N ,
\end{equation}
under these conditions.
%%  ENUMERATE
\begin{enumerate}
\item  $ N ^{1/r}$ satisfies the bounds of the Theorem.
\item The functions $ f$ and $ g$ are normalized so that
$ \lVert f\rVert _{L ^{p} (\sigma )} =1$ and $ \lVert g\rVert _{L ^{ q} (w)} =1$,
where $ q = (p/r)' = \frac p {p-r}$.
\item The sparse collection $ \mathcal S$ satisfies
\begin{equation}\label{e:Lb}
2 ^{a-1} < \langle \sigma  \rangle_ Q ^{p-1} \langle w \rangle_Q  \leq 2 ^{a},
\end{equation}
for some integer $ a$.   (And then one can sum over $ a$.)
\item All cubes $ Q\in \mathcal S$ are contained in a  root cube $ Q ^{0}$.
\end{enumerate}
%% ENUMERATE

The parallel corona is used to decompose the inner product in \eqref{e:La}.
Now we can define the principal cubes $\mathcal F$ for $(f,\sigma)$ and $\mathcal G$ for $(g, w)$. Namely,
\begin{eqnarray*}
\mathcal F&:=& \bigcup_{k=0}^\infty \mathcal F_k, \quad \mathcal F_0:=  \{\textup{maximal cubes in }\mathcal S\}\\
\mathcal F_{k+1}&:=& \bigcup_{F\in \mathcal F_k}\ch_{\mathcal F}(F),\quad \ch_{\mathcal F}(F):= \{ Q\subsetneq F\, \textup{maximal \,s.t.} \langle f\rangle_Q^\sigma>2\langle f\rangle_F^\sigma \},
\end{eqnarray*}
and analogously for $\mathcal G$. We also denote by  $\pi_{\mathcal F} (Q)$ the minimal cube in $\mathcal F$ which contains $Q$, and
$\pi (Q)=(F, G)$ if $\pi_{\mathcal F}(Q)=F$ and $\pi_{\mathcal G}(Q)=G$. With this definition, it is easy to check that for any $1<p<\infty$,
\begin{equation} \label{e:stop<}
\sum_{F\in \mathcal F} (\langle f\rangle_F^\sigma)^p \sigma(F) \lesssim \|f\|_{L^p(\sigma)}^p
\end{equation}
and a similar inequality holds for $g$.

In terms of the principal cubes, $ \langle (A ^{r} _{\mathcal S} f \sigma   ) ^{r}, g w \rangle  $ is
less than  $2^{r+1}$ times  the sum $ I + I\!I$, where
\begin{align}  \label{e:Idef}
I &:= \sum_{F\in \mathcal F}(\langle f\rangle_F^\sigma)^r
\underbrace{\sum_{\substack{G\in \mathcal G\\ \pi _{\mathcal F}(G)= F}}  \langle g\rangle_G^w
\sum_{\substack{Q\in \mathcal S\\ \pi(Q)=(F, G)}}  \langle\sigma\rangle_Q^r w(Q)}
_{:= I (F)},
\\ \label{e:IIdef}
I\!I := & \sum_{G\in \mathcal G}\langle g\rangle_G^w
\underbrace{
\sum_{\substack{F\in \mathcal F\\ \pi _{\mathcal G}(F)=G }} (\langle f\rangle_F^\sigma)^r
\sum_{\substack{Q\in \mathcal S\\ \pi(Q)=(F, G)}}   \langle\sigma\rangle_Q^r w(Q) } _{:= I\!I (G)} .
\end{align}
In view of \eqref{e:stop<},  the bound for $ I (F)$ we need is of the form below.
\begin{equation}\label{e:Lc}
I (F)
 \lesssim N _{I}  \sigma (F) ^{\frac rp}
 \Biggl[ \sum_{\substack{G\in \mathcal G\\ \pi _{\mathcal F}(G)= F}}  (\langle g\rangle_G^w) ^{q} w (G) \Biggr] ^{\frac 1q} ,
 \qquad F\in \mathcal F.
\end{equation}
Recall that  $ \frac 1q = \frac {p-r}p$, so that $ \frac 1q+ \frac rp=1$.
Indeed, with this bound, a straight forward application of H\"older's inequality, with \eqref{e:stop<}
completes a proof of \eqref{e:La}.
We then conclude that $ I  \lesssim N_I  $.
For $ I\!I (G)$, the bound is of the form below.
\begin{equation}\label{e:Ld}
I\!I (G)\lesssim N _{I\!I}  w (G) ^{\frac 1q}
 \Biggl[ \sum_{\substack{F\in \mathcal F\\ \pi _{\mathcal G}(F)= G}}  (\langle f\rangle_F^ \sigma ) ^{r} \sigma  (F) \Biggr]  ^{\frac rp}, \qquad G\in \mathcal G.
\end{equation}

\medskip

Let us now bound $ N _{I}$, for all $r<p<\infty$.
Observe that  
\begin{eqnarray*}
I (F)&\lesssim&
\sum_{\substack{G\in \mathcal G\\ \pi_{\mathcal F}(G)=F}}\sum_{\substack{Q\in \mathcal S\\ \pi(Q)=(F, G)}}\int_Q\Bigl(\sup_{\substack{G'\in \mathcal G \\ \pi _{\mathcal F} (G')=F }}\langle g\rangle_{G'}^w\mathbf 1_{G'}\Bigr)
\langle\sigma\rangle_Q^r \mathbf 1_Q \dw\\
&\le& \int_F\Bigl(\sup_{\substack{G'\in \mathcal G \\ \pi _{\mathcal F} (G')=F }}\langle g\rangle_{G'}^w\mathbf 1_{G'}\Bigr)
A^r_{\mathcal S(F)}(\sigma \mathbf 1_F)^r \dw\\
&\le&
\Big\|\sup_{G'\in \mathcal G(F)}\langle g\rangle_{G'}^w\mathbf 1_{G'}\Big\|_{L^{q}(w)}
\cdot
\|A^r_{\mathcal S(F)}(\sigma \mathbf 1_F)  ^{r}\|_{L^{p/r}(w)}.
\end{eqnarray*}
Now, the first term on the right, by construction of the principal cubes is no more than
\begin{equation*}
 \Biggl[ \sum_{\substack{G\in \mathcal G\\ \pi _{\mathcal F}(G)= F}}  (\langle g\rangle_G^w) ^{q} w (G) \Biggr] ^{\frac 1q}
\end{equation*}
as required in \eqref{e:Lc}.  The second term, the notation is $ \mathcal S (F) = \{Q\in \mathcal S \;:\; \pi _{\mathcal F}(Q)=F\} $.
Below, we dominate $ \ell ^{r}$-norms by $ \ell ^{1}$, and appeal to  \cite[Prop. 5.3]{HL} to see that
\begin{align*}
\|A^r_{\mathcal S(F)}(\sigma \mathbf 1_F)  ^{r}\|_{L^{p/r}(w)}
& \leq \|A^1_{\mathcal S(F)}(\sigma \mathbf 1_F)  \|_{L^{p}(w)} ^{r}
 \lesssim \bigl[ [w, \sigma ] _{A_p,\mathcal S(F)} [\sigma ] _{A _{\infty },\mathcal S(F)} \sigma (F) \bigr] ^{\frac rp}.
\end{align*}
Our conclusion is that
\begin{equation}\label{e:NI}
N _{I} (F) \lesssim 2 ^{a \frac rp}    [\sigma ] _{A _{\infty }}  ^{\frac rp} .
\end{equation}

\medskip
Now we turn to the analysis of $ I\!I (G)$. This is the more delicate case, that breaks into the two subcases of
$ r< p \leq  r+1$, and $ r+1\leq p$, though the resulting inequality is the same in both cases.
We treat the case of $ r<p< r+1$ first.
The first step is to again appeal to \eqref{e:Lb} to write
\begin{align*}
 \langle\sigma\rangle_Q^r w(Q)
& \simeq  \langle\sigma\rangle_Q^r \langle w \rangle_Q \cdot \lvert  Q\rvert
\\
& \simeq 2 ^{a} \langle\sigma\rangle_Q^ {r+1-p} \lvert  Q\rvert.
%\\
%& \simeq 2 ^{a} \sigma (Q) ^{r+1-p} \lvert  Q\rvert ^{p-r} \simeq 2 ^{a} \langle \sigma  \rangle_Q ^{r+1-p} \lvert  Q\rvert.
\end{align*}
Indeed, in the last line, we can replace $ \lvert  Q\rvert $ by $ \lvert  E (Q)\rvert $, the execptional set
associated to $ Q$. The sets $ E (Q)$ are disjoint in $ Q$, whence
\begin{align*}
\sum_{\substack{Q\in \mathcal S \\ \pi(Q)=(F,G)}} \langle\sigma\rangle_Q^r w(Q)
& \lesssim 2^a \int _{F} M (\sigma \mathbf 1_{F}) ^{r+1-p} \; dx .
\end{align*}
Now, recall that on probability spaces that $ L ^{t}$ norms increase in $ t$. This has an extension to
Lorentz spaces, from which we conclude that
\begin{align*}
\Biggl[ \frac 1 {\lvert  F\rvert } \int _{F} M (\sigma \mathbf 1_{F}) ^{r+1-p} \; dx \Biggr] ^{\frac 1 {r+1-p}}
& \leq
\lVert  \mathbf 1_{F} M (\sigma \mathbf 1_{F}) \rVert _{L ^{1, \infty }  ( F, \frac {dx} {\lvert  F\rvert })}
\\
& \lesssim \langle \sigma  \rangle_F.
\end{align*}
This just depends upon the maximal function bound.   Simplifying, we have the bound
\begin{align*}
\sum_{\substack{Q\in \mathcal S \\ \pi(Q)=(F,G)}} \langle\sigma\rangle_Q^r w(Q)
& \lesssim 2^a \sigma (F) ^{r+1-p } \lvert  F\rvert ^{p-r}
\\
& \lesssim  2^a 2 ^{a \frac {r-p}p} \sigma (F) ^{r+1-p} \bigl[ \sigma (F) ^{p-1}w (F) \bigr] ^{\frac 1q }.
\end{align*}
Here, we must note that the power on $ \sigma (F)$ is  ($ \frac 1q = \frac  {p-r}p$)
\begin{align*}
r+1-p  + (p-1) \frac {p-r} p = \frac rp .
\end{align*}
Indeed, this is easy to see by multiplying both sides above  by $ p$.

It follows that
\begin{align*}
I\!I (G)  & \lesssim 2 ^{a \frac rp }
  \sum_{\substack{F\in \mathcal F\\ \pi _{\mathcal G} (F)=G}}
(\langle f \rangle ^{\sigma } _{F} ) ^ r   \sigma (F)  ^{ \frac rp }  w (F)  ^{\frac 1q}
\\
&  \lesssim  2 ^{a \frac rp }    [w] _{A _{\infty   }} ^{\frac  {p-r}p} w(G)^{\frac 1 q}\Biggl[ \sum_{\substack{F\in \mathcal F\\ \pi _{\mathcal G}(F)= G}}  (\langle f\rangle_F^ \sigma ) ^{p} \sigma  (F) \Biggr]  ^{\frac rp},
\qquad r < p < r+1.
\end{align*}
This just depends upon an application of H\"older's inequality, and an appeal to the $ A _{\infty }$ constant of $ w$
to bound the sum over $ F$ of $ w (F)$.
We conclude the bound below, which is just as the Theorem claims.
\begin{equation} \label{e:II1}
N _{I\!I} \lesssim  2 ^{a \frac rp }  [w] _{A _{\infty   }} ^{\frac  {p-r}p},
\qquad r < p < r+1.
\end{equation}

\medskip

The last case is to estimate $ N _{I\!I}$ in the case of $ r+1\leq p < \infty $.
We begin by eliminating the $ \sigma (Q) ^{r} $ in the term on the right below: By \eqref{e:Lb}
\begin{equation*}
 \langle\sigma\rangle_Q^r w(Q) \simeq
2^{\frac {ar}{p-1}}w(Q)^{1-\frac r{p-1}}|Q|^{\frac r{p-1}}.
\end{equation*}
The exponents above are in H\"older's duality, thus
\begin{align*}
\sum_{Q \;:\; \pi (Q) = (F,G)}
 \langle\sigma\rangle_Q^r w(Q)
 & \lesssim 2^{\frac {ar}{p-1}}  \lvert  F\rvert ^{\frac r{p-1}}
 \Biggl[   \sum_{Q \;:\; \pi (Q) = (F,G)}   w (Q)\Biggr] ^{1-\frac r{p-1}}
 \\
 & \lesssim 2 ^{a\frac rp} \sigma (F) ^{\frac rp}
  \Biggl[   \sum_{Q \;:\; \pi (Q) = (F,G)}   w (Q)\Biggr] ^{1-\frac r{p-1}}  w (F) ^{\frac r {p (p-1)}}.
\end{align*}
In the second line, appeal to  \eqref{e:Lb} again, converting  $ \lvert  F\rvert $ in the first  line into a geometric mean
of $ \sigma (F)$ and $ w (F)$.

Therefore, from the definition of $ I\!I (G)$, we have
\begin{align*}
I\!I (G) & \lesssim 2 ^{a\frac rp}
\sum_{F \;:\; \pi _{\mathcal G} (F) = G}
(\langle f \rangle_F ^{\sigma }) ^{r} \sigma (F) ^{\frac rp}
  \Biggl[   \sum_{Q \;:\; \pi (Q) = (F,G)}   w (Q)\Biggr] ^{1-\frac r{p-1}}  w (F) ^{\frac r {p (p-1)}}
  \\
  & \lesssim
  2 ^{a\frac rp}   [w] _{A _{\infty }} ^{\frac 1q}
  \Biggl[   \sum_{F \;:\; \pi _{\mathcal G} (F) = G}
(\langle f \rangle_F ^{\sigma }) ^{p} \sigma (F)  \Biggr] ^{\frac rp} w (G) ^{\frac 1q}
\end{align*}
after an application of the trilinear form of H\"older's inequality, and the use of the $ A _{\infty }$ property of $ w$.
We conclude that
\begin{equation}\label{e:II2}
N _{I\!I}  \lesssim  2 ^{a\frac rp}   [w] _{A _{\infty }} ^{\frac {p-r} p} , \qquad r+1\leq p < \infty .
\end{equation}
The proof follows by combining the estimates \eqref{e:Lc}, \eqref{e:NI},  \eqref{e:Ld}, \eqref{e:II1} and \eqref{e:II2}.
\end{proof}
%%%%%%%%%%%%%%%%%%%%%%%%%%%%%% PROOF PROOF PROOF

The case of $1<p\le r$ is quite simple, we have the following estimate
\begin{Theorem}\label{thm:r}
Let $1<p\le r$. There holds
\[
\|A_{\mathcal S}^r(\cdot\sigma)\|_{L^p(\sigma)\rightarrow L^p(w)}\le C [w,\sigma]_{A_p}^{\frac 1p} [\sigma]_{A_\infty}^{\frac 1 p},
\]
where the constant $C$ is independent of the weights $w$ and $\sigma$.
\end{Theorem}

\begin{proof}
We only need principle cubes  for the function $ f$, and we use the same definition and notation from the previous proof.
Since $ \ell ^{p}$ norms are larger than $ \ell ^{r}$ norms, we have
\begin{eqnarray*}
\|A_{\mathcal S}^r(f\sigma) \|_{L^p(w)} ^{p}
&=& \bigg\|\bigg(\sum_{F\in \mathcal F}\sum_{\substack{Q\in \mathcal S\\ \pi(Q)=F}}  (\langle f\rangle_Q^\sigma) ^r\langle\sigma\rangle_Q^r \mathbf 1_Q\bigg)^{\frac 1 r}\bigg\|_{L^p(w)} ^{p}\\
& \lesssim  &  \sum_{F\in\mathcal F} (\langle f\rangle_F^\sigma)^p
\sum_{\substack{Q\in \mathcal S\\ \pi(Q)=F}}  \langle \sigma  \rangle_Q ^{p} w (Q)
\\
& \lesssim &  [w,\sigma]_{A_p}
\sum_{F\in\mathcal F} (\langle f\rangle_F^\sigma)^p
\sum_{\substack{Q\in \mathcal S\\ \pi(Q)=F}}  \sigma  (Q)
\\
&\lesssim& [w,\sigma]_{A_p} [\sigma]_{A_\infty} \|f\|_{L^p(\sigma)} ^{p}
\end{eqnarray*}
where we have used the $ A_p$ and $ A _{\infty }$ properties in a straight forward way.
\end{proof}
Now with Theorems~\ref{thm:p} and \ref{thm:r}, Theorem~\ref{thm:m1} follows immediately, by setting $ r=2$.

%%%%%%%%%%%%%%   Section
\section{Proof of Theorem~\ref{thm:m2}}
In this section, we shall give a proof for Theorem~\ref{thm:m2}. In the proof, we will have
recourse to this Carleson embedding inequality.  Proved in \cite[Theorem 4.2]{TV}, it has a very short proof given in \cite[\S4]{LS}.

\begin{Lemma}\label{lm:ls}
Let $\epsilon$ be a monotonic increasing function on $(1,\infty)$ such that $\int_1^\infty \frac{dt}{\epsilon(t) t}<\infty$. Then for all $1<p<\infty$, we have
\[
\sum_{Q\in \mathcal S} (\langle f\rangle_Q^\sigma)^p \frac{\sigma(Q) }{\rho_{\sigma, \epsilon}(Q) } \lesssim \|f\|_{L^p(\sigma)}^p.
\]
\end{Lemma}

Again, we reduce the problem to consider the sparse operators and we shall show the following more general estimate.
\begin{Theorem}\label{thm:ent}
Let $A_{\mathcal S}^r$ be defined as in \eqref{eq:t} and $\epsilon,\eta$ be two monotonic increasing functions on $(1,\infty)$. Then if $p> r$  and $\int_1^\infty \frac{dt}{\epsilon(t) t}<\infty$, we have the entropy bound
\begin{equation}\label{eq:ent}
\|A_{\mathcal S}^r(\cdot \sigma)\|_{L^p(\sigma)\rightarrow L^p(w)}\lesssim
\sup_Q\langle w\rangle_Q^{\frac 1p} \langle\sigma\rangle_Q^{\frac 1{p'}}\rho_{w,\epsilon}(Q)^{\frac 1r-\frac 1p}\rho_{\sigma,\epsilon}(Q)^{\frac 1p}.
\end{equation}
If $p> r$ and $\epsilon,\eta$ satisfy
\[
\int_1^\infty \frac 1{t\epsilon(t)^{1/p}}dt+ \int_1^\infty \frac 1{t\eta(t)^{\frac 1r-\frac 1p}}dt<\infty,
\]
we also have the separated entropy bound
\begin{equation}\label{eq:sep}
\|A_{\mathcal S}^r(\cdot \sigma)\|_{L^p(\sigma)\rightarrow L^p(w)}\lesssim
\sup_Q\langle w\rangle_Q^{\frac 1p} \langle\sigma\rangle_Q^{\frac 1{p'}}(\rho_{w,\eta}(Q)^{\frac 1r-\frac 1p}+\rho_{\sigma,\epsilon}(Q)^{\frac 1p}).
\end{equation}
\end{Theorem}
\begin{proof}
First, we prove \eqref{eq:ent}. Denote
\[
\lceil w,\sigma\rceil_{p,r, \epsilon}=\sup_Q\langle w\rangle_Q^{\frac 1p} \langle\sigma\rangle_Q^{\frac 1{p'}}\rho_{w,\epsilon}(Q)^{\frac 1r-\frac 1p}\rho_{\sigma,\epsilon}(Q)^{\frac 1p}.
\]
 Follow the method used in \cite{LS}, set
 \[
 \mathcal Q_a=\{Q\in\mathcal S: 2^a< \langle w\rangle_Q^{\frac 1p} \langle\sigma\rangle_Q^{\frac 1{p'}}\rho_{w,\epsilon}(Q)^{\frac 1r-\frac 1p}\rho_{\sigma,\epsilon}(Q)^{\frac 1p}\le 2^{a+1}\},
 \]
 here $2^a\le \lceil w,\sigma\rceil_{p,r,\epsilon}$. Recall that $q=(p/r)'$, by duality, we have
 \begin{eqnarray*}
 \|A_{\mathcal S}^r(f\sigma)\|_{L^p(w)}^r
&=& \sup_{\|g\|_{L^{q'}(w)}=1} \sum_{Q\in \mathcal S} \langle f\sigma\rangle_Q^r\int_Q g \dw.
 \end{eqnarray*}
 Now fix $g\in L^{q'}(w)$ with $\|g\|_{L^{q'}(w)}=1$. We have for $ Q\in \mathcal Q_a$,
 \begin{align*}
\langle f\sigma\rangle_Q^r\int_Q g \dw
&= (\langle f\rangle_Q^\sigma)^r \langle g\rangle_Q^w \langle \sigma\rangle_Q^r w(Q)
\\
& \lesssim
2 ^{ar}
 (\langle f\rangle_Q^\sigma)^r \frac{\sigma(Q)^{\frac rp}}{\rho_{\sigma, \epsilon}(Q)^{\frac rp}}
 \cdot
 \langle g\rangle_Q^w
 \frac{w(Q)^{\frac 1q}}{\rho_{w,\epsilon}(Q)^{\frac 1q}} .
\end{align*}
 The indices are set up for an application of H\"older's inequality.  The sum over $ Q$ of the terms above is
 \begin{align*}
&\sum_{a\le \log_2 \lceil w,\sigma\rceil_{p,r,\epsilon} } \sum_{Q\in \mathcal Q_a} \langle f\sigma\rangle_Q^r\int_Q g \dw
\\
& \lesssim \sum_{a\le \log_2 \lceil w,\sigma\rceil_{p,r,\epsilon} }2^{ar}\biggl[ \sum_{Q\in \mathcal Q_a} (\langle f\rangle_Q^\sigma)^p \frac{\sigma(Q) }{\rho_{\sigma, \epsilon}(Q) } \biggr]^{r/p} \biggl[\sum_{Q\in \mathcal Q_a} (\langle g\rangle_Q^w)^{q'} \frac{w(Q) }{\rho_{w,\epsilon}(Q) } \biggr]^{\frac 1{q}}
\\
 &\lesssim \lceil w, \sigma\rceil_{p,r,\epsilon}^r \|f\|_{L^p(\sigma)}^r.
\end{align*}
Lemma~\ref{lm:ls} is used in the last step, to control  the sums involving  both $ f$ and  $ g$.

 Next we consider \eqref{eq:sep}.
We decompose the collection $ \mathcal S$ into subsets  $\mathcal S_{a,b}$ and $ \mathcal S' _{a,b}$, for integers $ a, b$.
The collection $ \mathcal S _{a,b}$ consists of those $Q\in \mathcal S$ which meet three conditions,
\begin{gather*}
2^a< \langle w\rangle_Q^{\frac 1p} \langle\sigma\rangle_Q^{\frac 1{p'}}
 \rho_{\sigma,\epsilon}(Q)^{\frac 1p} \le 2^{a+1}
\\
\rho_\sigma(Q)^{\frac 1p}\ge \rho_w(Q)^{\frac 1r-\frac 1p},
\\
2^b< \rho_\sigma(Q) \le 2^{b+1}.
\end{gather*}
We also denote $\mathcal S_{a,b}'$ the sub-collection of $\mathcal S$ such that
\begin{gather*}
2^a< \langle w\rangle_Q^\frac 1p \langle\sigma\rangle_Q^{\frac 1{p'}}  \rho_{w,\eta}(Q)^{\frac 1r-\frac 1p} \le 2^{a+1},
\\
\rho_\sigma(Q)^{\frac 1p}< \rho_w(Q)^{\frac 1r-\frac 1p},
\\
2^b< \rho_w(Q) \le 2^{b+1}.
\end{gather*}
Every  $Q\in \mathcal S$ is in either $ \mathcal S _{a,b}$ or $ \mathcal S _{a,b}'$
 for some choice of $ a,b$, and are empty  if either  $ b \leq -1 $, or $ 2 ^{a} > 2 \lfloor w,\sigma\rfloor_{p,r,\epsilon,\eta}$.

The initial estimate is then as below, where it is important to note that we are using the quantification of
Theorem~\ref{thm:p}, as described in \eqref{e:reduce}.
\begin{eqnarray*}
\|A_{\mathcal S}^r(f\sigma)\|_{L^p(w)} &\le& \sum_{a,b} \|A_{\mathcal S_{a,b}}^r(f\sigma)\|_{L^p(w)}+ \|A_{\mathcal S_{a,b}'}^r(f\sigma)\|_{L^p(w)} \\
&\le& \sum_{a,b} [w, \sigma ] _{A_p, \mathcal S _{a,b}}^{\frac 1p}
\bigl([ w] _{A _{\infty }, \mathcal S _{a,b}}^{\frac 1r-\frac 1p}
+[ \sigma ] _{A _{\infty }, \mathcal S _{a,b}}^{\frac 1p}\bigr)\\
&& + \sum_{a,b} [w, \sigma ] _{A_p, \mathcal S _{a,b}'}^{\frac 1p}\bigl([ w] _{A _{\infty }, \mathcal S _{a,b}'}^{\frac 1r-\frac 1p}+[\sigma] _{A _{\infty }, \mathcal S _{a,b}'}^{\frac 1p}\bigr)\\
&:=& I+I\!I.
\end{eqnarray*}

First, we estimate $I$. By definition of $ \mathcal S _{a,b}$,  there holds
\begin{equation*}
[ w] _{A _{\infty }, \mathcal S _{a,b}}^{\frac 1r-\frac 1p}
\leq [ \sigma ] _{A _{\infty }, \mathcal S _{a,b}}^{\frac 1p},
\end{equation*}
so that
\begin{eqnarray*}
I&\lesssim& \sum_{a,b}
 [w, \sigma ] _{A_p, \mathcal S _{a,b}}^{\frac 1p} [ \sigma ] _{A _{\infty }, \mathcal S _{a,b}}^{\frac 1p}
 \\
&\lesssim& \sum_{a,b} \frac {2^{a}}{2^{b/p}\epsilon(2^b)^{1/p}} 2^{b/p}
\\
&\lesssim&   N\int_1^\infty \frac 1{t\epsilon(t)^{1/p}}dt,
\end{eqnarray*}
where $  N$ denotes the right side of \eqref{eq:sep}.
The dual term follows an analogous line of reasoning, leading to the estimate below, which completes the proof.
\[
I\!I\lesssim N \int_1^\infty \frac 1{t\eta(t)^{\frac 1r-\frac 1p}}dt.
\]
\end{proof}
It remains to consider the case $1<p\le r$. We have the following result.
\begin{Theorem}\label{thm:ent1}
Let $A_{\mathcal S}^r$ be defined as in \eqref{eq:t} and $\epsilon$ be a monotonic increasing function on $(1,\infty)$ such that $\int_1^\infty \frac 1{t\epsilon(t)^{1/p}}dt<\infty$. If $1<p\le r$, then
\[
\|A_{\mathcal S}^r(\cdot \sigma)\|_{L^p(w)\rightarrow L^p(\sigma)}\lesssim \sup_Q \langle w\rangle_Q^{\frac 1p} \langle\sigma\rangle_Q^{\frac 1{p'}}\rho_{\sigma,\epsilon}(Q)^{\frac 1p}.
\]
\end{Theorem}
\begin{proof}
From the proof of Theorem~\ref{thm:r}, we know that
\[
\|A_{\mathcal S}^r(\cdot \sigma)\|_{L^p(w)\rightarrow L^p(\sigma)}\le \sup_{R\in\mathcal S} \frac{\Big\| \sum_{\substack{ Q\in \mathcal S\\ Q\subset R }} \langle\sigma\rangle_Q \mathbf 1_Q\Big\|_{L^p(w)}}{\sigma(R)^{1/p}}.
\]
Then by the same argument as that in \cite{LS}, we can get the conclusion.
\end{proof}
Now with Theorems~\ref{thm:ent} and \ref{thm:ent1}, Theorem~\ref{thm:m2} follows.

\end{document}